\documentclass{article}
\usepackage{amsmath}
\usepackage{amssymb}
\usepackage{amsfonts}
\usepackage{amsthm}

\newtheorem{theorem}{Theorem}
\newtheorem{corollary}{Corollary}
\newtheorem{note}{Note}
\newtheorem{proposition}{Proposition}

\title{ Generalized Compositions and Weighted Fibonacci Numbers}
\author{{Milan Janji\'c}}%
\date{}
\begin{document}
\maketitle
\begin{center}
Department of Mathematics and Informatics\\
 University of Banja Luka\\
78000 Banja Luka, Republic of Srpska, BA\\
email: agnus@blic.net
\end{center}

\begin{abstract}
In this paper we consider particular generalized compositions of a natural number with a given number of parts. Its number is a weighted polynomial coefficient. The number of all generalized compositions of a natural number is a weighted $r$-generalized Fibonacci number. A relationship between these two numbers will be derived.

We shall thus obtain a generalization of the well-known formula connecting Fibonacci numbers with the binomial coefficients.
\end{abstract}

\section{Introduction}
 Weighted $r$-generalized Fibonacci numbers and weighted polynomial coefficients are defined by M. Schork, \cite{sh}, to be a generalization of the well-known formula that expresses Fibonacci numbers in terms of the binomial coefficients. This formula
  may also be interpreted as a connection between all compositions of a natural number $n,$  of a certain kind, and the number of compositions of $n$ of the same kind with a fixed number of parts.  Namely, Fibonacci numbers count the compositions of the natural numbers in which each part is either $1$ or $2.$ On the other hand, we shall show that the binomial coefficients count the compositions of the same kind into a given number of parts.

Generalized compositions are considered by M. Janji\'c, in \cite{mil}.
They are defined as the compositions when there are different types of each natural number.
In this paper we consider the  generalized compositions with a given number of parts.
We show that the weighted $r$-generalized Fibonacci numbers count the numbers of all generalized  compositions. Also, the weighted polynomial coefficients count the number of generalized compositions of the same kind into a given number of parts.

A formula which connects the weighted $r$-generalized  Fibonacci numbers with
the weighted polynomial coefficients is derived. The formula extends the above mentioned relationship between Fibonacci numbers and the binomial coefficients.

Note that our results are self contained, depending neither on the results  in \cite{mil} nor on the results in \cite{sh}.
\section{Generalized Compositions, Weighted Fibonacci Numbers and Weighted Polynomial Coefficients}
Let $\mathbf b=(b_1,b_2,\ldots,b_r)$ be a finite sequence of nonnegative integers, and let $k,n$ be positive integers.
The compositions of  $n$ into $k$ parts, under the condition that there are $b_1$ different types of  $1$, $b_2$ different types of  $2$, and so on, will be called generalized compositions of $n$ into $k$ parts.
We let  $C^{(r)}(k,n)$ denote
its number.

In the following proposition we state some simple properties of such compositions.
\begin{proposition}\label{pp1} The following equations are true:
   \[C^{(r)}(1,i)=b_i,\;(i=1,2,\ldots,r),\;C^{(r)}(n,n)=b_1^n,\]\[C^{(r)}(k,n)=0,\;(k\cdot r<n),\;C^{(r)}(k,n)=0,\;(k>n).\]
\end{proposition}
\begin{proof} All equations are easy to verify.
\end{proof}
\begin{proposition} The following recursion is true:
\begin{equation}\label{e1}C^{(r)}(k,n)=\sum_{i=1}^{\min\{r,n-k+1\}}b_{i}C^{(r)}(k-1,n-i),\;(k\leq n\leq kr).\end{equation}
\end{proposition}
\begin{proof} Equation (\ref{e1}) is true since there are $b_iC^{(r)}(k-1,n-i),\;(i=1,\ldots,n-k+1)$ generalized compositions ending with one of the $i$'s.
\end{proof}

Following \cite{sh}, we denote
\[\big(b_1+b_2x+b_x^2+\cdots+b_rx^{r-1}\big)^k=\sum_{i=0}^{(r-1)k}{k,r|\mathbf b\choose i}x^i,\]
where ${k,r|\mathbf b\choose i}$ are called  weighted
polynomial coefficients.

\begin{theorem}\label{t1} The following formula is true.
\[ C^{(r)}(k,n)={k,r|\mathbf b\choose n-k},\;(k\leq n\leq
rk),\]
\end{theorem}
\begin{proof} We use induction with respect to $k.$
For $k=1$ we have
\[b_1+b_2x+b_x^2+\cdots+b_rx^{r-1}={1,r|\mathbf b\choose 0}+{1,r|\mathbf b\choose 1}x+\cdots+{1,r|\mathbf b\choose r-1}x^{r-1}.\]
According to Proposition \ref{pp1} we have $c^{(r)}(1,i)=b_i,\;(1\leq i\leq r),$ therefore
\[c^{(r)}(1,i)={1,r|\mathbf b\choose i-1},\;(1\leq i\leq r),\] and the assertion is true for $k=1.$

Assume that $k>1,$ and that  the
assertion is true for $k-1.$
According to the obvious fact that
\[\big(\sum_{i=0}^{r-1}b_{i+1}x^i\big)^k=\big(\sum_{i=0}^{r-1}b_{i+1}x^i\big)\big(\sum_{i=0}^{r-1}b_{i+1}x^i\big)^{k-1},\]
and using the induction hypothesis we may write equation (\ref{e1}) in the form:
\begin{equation}\label{e3}\sum_{i=0}^{(r-1)k}{k,r|\mathbf b\choose i}x^i=\sum_{i=0}^rb_{i+1}x^i\cdot
 \sum_{i=0}^{(r-1)(k-1)}C^{(r)}(k-1,k+i-1)x^i.\end{equation}

Case 1. $n\geq r-1.$

Comparing the terms of $x^n$ in (\ref{e3}) we obtain
\[{k,m|\mathbf b\choose n}=b_1c^{(r)}(k-1,k+n-1)+\cdots+b_rC^{(r)}(k-1,k+n-r).\]
According to (\ref{e1}) we have  ${k,r|\mathbf b\choose
n}=C^{(r)}_{k}(n+k,\mathbf b).$

Case 2. $n<r-1.$ We again compare the terms of $x^n$ in (\ref{e3}) to obtain
\[{k,r|\mathbf b\choose i}=b_1C^{(r)}(k-1,k+n-1)+\cdots+b_{n+1}C^{(r)}(k-1,k-1),\]
and the assertion is again true according to (\ref{e1}).
\end{proof}

If $\mathbf b=(1,1,0,0,\ldots,0),$ then $C^{(2)}(k,n)$ is the number of standard compositions of $n$ into $k$ parts which are either $1$ or $2.$ We would like to show that this number is a binomial coefficient.
\begin{theorem}\label{c1}
Let $k$ be a positive integer, let $n$ a nonnegative  integer, and let $\mathbf b=(1,1,0,0,\ldots).$
Then
\[C^{(2)}(k,n+k)={k\choose n},\;(0\leq n\leq k).\]
\end{theorem}
\begin{proof}
We firstly have
\[C^{(2)}_1(1)=1,\;C^{(2)}_1(2)=1.\]
Using induction, the recursion
\[C^{(2)}(k,k+n)=C^{(2)}(k-1,k+n-1)+C^{(2)}(k-1,k+n-2),\;(1\leq n\leq k),\]
becomes
\[{k\choose n}={k-1\choose n}+{k-1\choose n-1},\]
which is the well-known recursion for the binomial coefficients.
\end{proof}

It is a well-known fact that $F_{n+1}$ is the number of all
compositions of $n$ into parts which are either $1$ or $2.$ This
fact and the preceding corollary give a natural connection between
Fibonacci numbers and the binomial coefficients. Namely, we have the following well-known result:
\begin{theorem}\label{t3} Let $n$ be a positive integer. Then
\begin{equation}\label{e4}F_{n+1}=\sum_{i=\lceil\frac n2\rceil}^n{i\choose n-i}.\end{equation}
\end{theorem}
\begin{proof}
The assertion follows from Corollary \ref{c1} and the fact that:
\[F_{n+1}=\sum_{k=1}^nC^{(2)}(k,n)),\;(k\leq n\leq 2k).\]
\end{proof}

Now, we are going to generalize this formula on the $r$- generalized Fibonacci numbers and the weighted polynomial coefficients.

Let $\mathbf b=(b_1,b_2,\ldots,b_r)$ be a sequence of nonnegative integers.  We let $F^{(r)\mathbf b}_{n}$ denote the number of all generalized compositions of $n.$ Additionally, we put $F^{(r)\mathbf b}_{0}=1.$
\begin{proposition}\label{p1} Let $n,r$ be positive integers.

If  $n\geq r,$ then
 \[F^{(r)\mathbf b}_{n}=b_1F^{(r)\mathbf b}_{n-1}+b_2F^{(r)\mathbf b}_{n-2}+\cdots+b_rF^{(r)\mathbf b}_{n-r}.\]

If  $n<r,$ then
\[F^{(r)\mathbf b}_{n}=b_1F^{(r)\mathbf b}_{n-1}+b_2F^{(r)\mathbf b}_{n-2}+\cdots+b_nF^{(r)\mathbf b}_{0}.\]
\end{proposition}
\begin{proof} Both formulas are true according to the fact that there are $b_1F^{(r)\mathbf b}_{n-1}$ generalized compositions ending by one of the $1$'s, and so on.
\end{proof}
In \cite{sh} the numbers $F^{(r)\mathbf b}_{n}$ are called   $r$-generalized Fibonacci numbers.

Adding the generalized compositions into $1,2,\ldots$ parts, we obtain all generalized compositions. This fact and
Theorem \ref{t1} imply
\begin{theorem}\label{t2} Let $n,r$ be positive integers. Then
 \[F^{(r)\mathbf b}_{n}=\sum_{k=\lceil\frac{n}{r}\rceil}^n{k,r|\mathbf b\choose n-k}.\]
\end{theorem}
\begin{proof} The lower bound of the index $k$ is $\lceil\frac{n}{r}\rceil$ since $n\leq rk.$ The rest is clear.
\end{proof}
\begin{note} The preceding equation suits equation $(18)$ in \cite{sh}.
\end{note}
 We shall now show that Theorem \ref{t3} is a consequence of Theorem \ref{t2}.
Take $r=2,$ and $b_1=b_2=1.$ Then the formulas  in Proposition \ref{p1} have the form:
\[F^{(2)\mathbf b}_{n}=F^{(2)\mathbf b}_{n-1}+F^{(2)\mathbf b}_{n-2},\]
with the initial conditions given by
\[F^{(2)\mathbf b}_{0}=1,\;F^{(2)\mathbf b}_{1}=1.\]
This is the recursion for Fibonacci numbers, shifted for $1$ forward. Since ${k,2|\mathbf b\choose n-k}={k\choose n-k}$ we
conclude that Theorem \ref{t3} is a consequence of Theorem \ref{t2}.

 Take
$b_1=b_2=\cdots=b_r=1$ in  the conditions of Theorem \ref{t2}. Then  ${k,r|\mathbf b\choose i}$ are the polynomial
coefficients, and are usually denoted by ${k,r \choose n }.$
We thus have
\begin{corollary} The polynomial coefficient ${k,r \choose n-k }$ is equal to the number standard composition of $n,\;(k\leq n\leq rk)$
 into $k$ parts, all of which are $\leq r.$
\end{corollary}
We finish the paper with a formula connecting the $r$-generalized Fibonacci numbers with the polynomial coefficients.

If $b_1=b_2=\cdots=b_r=1$, the recursions from Proposition \ref{p1} have the form

 \[F^{(r)\mathbf b}_{n}=F^{(r)\mathbf b}_{n-1}+F^{(r)\mathbf b}_{n-2}+\cdots+F^{(r)\mathbf b}_{n-r},\]
 for $n\geq r,$ and
\[F^{(r)\mathbf b}_{n}=F^{(r)\mathbf b}_{n-1}+F^{(r)\mathbf b}_{n-2}+\cdots+1F^{(r)\mathbf b}_{0},\]
for $n<r.$  It follows that the initial conditions are
\[F^{(r)\mathbf b}_{0}=1,\;F^{(r)\mathbf b}_{1}=1,\ldots,;F^{(r)\mathbf b}_{r-1}=F_r.\]
Therefore, we have the recursion for $r$-generalized Fibonacci numbers $F^{(r)}_i,$ without leading zeroes.

According Theorem \ref{t2} we have
\begin{corollary} Let $n,r$ be positive integers. Then
\[F^{(r)}_{n+r-1}=\sum_{k=\lceil\frac nr\rceil}^n{k,r \choose n-k }.\]
\end{corollary}
\begin{note} The preceding equation suits  equation $(9)$ in \cite{sh}.
\end{note}

\end{document}